\documentclass[10pt,leqno]{amsart}%
\usepackage{amsmath}
\usepackage{amsfonts}
\usepackage{amssymb}
\usepackage{graphicx}%
\usepackage{color}
\usepackage{hyperref}

\providecommand{\U}[1]{\protect\rule{.1in}{.1in}}
\newtheorem{theorem}{Theorem}

\newtheorem{corollary}[theorem]{Corollary}

\newtheorem*{definition}{Definition}

\newtheorem{lemma}[theorem]{Lemma}

\newtheorem{proposition}[theorem]{Proposition}
\newtheorem{remark}[theorem]{Remark}

\newcommand{\dive}{\operatorname{div}}
\newcommand{\Capp}[1]{\operatorname{Cap}_{#1}}
\newcommand{\R}{\mathbb{R}}
\newcommand{\abs}[1]{\left|#1\right|}
\newcommand{\ps}[2]{\left\langle#1,#2\right\rangle}
\newcommand{\ton}[1]{\left(#1\right)}

\begin{document}

 \title{Stokes' theorem, volume growth and parabolicity}

 \author{Daniele Valtorta}
 \address{Dipartimento di Matematica\\
 Universit\`a degli Studi di Milano\\
 via Saldini 50\\
 I-20133 Milano, Italy, European Union.} \email{danielevaltorta@gmail.com}
 
 \author{Giona Veronelli}
 \address{Dipartimento di Matematica\\
 Universit\`a degli Studi di Milano\\
 via Saldini 50\\
 I-20133 Milano, Italy, European Union.} \email{giona.veronelli@unimi.it}

\begin{abstract}
%In this paper we present a Stokes' type theorem for vector fields $X$ on Riemannian manifolds with restrictions on the growth at infinity of 
%$\abs X ^{\frac p {p-1}}$ 
%related to the growth of special exhaustion functions involving volume estimates and parabolicity of the underlying manifold. These new techniques 
%lead to a generalization of the Kelvin Nevanlinna Royden criterion for $p$-parabolicity, which in turn can be used to extend some $p$-laplacian 
%comparisons and uniqueness results for $p$-harmonic functions and maps in homotopy class. In particular, we prove a $p$-laplacian comparison for 
%vector valued maps, also by weakening, with respect to the previous known results, the regularity assumptions on the maps.
We present some new Stokes' type theorems on complete non-compact manifolds that extend, in different directions, previous work by Gaffney and Karp and also the so called Kelvin-Nevanlinna-Royden criterion for $p$-parabolicity. Applications to comparison and uniqueness results involving the $p$-Laplacian are deduced.
\end{abstract}

\maketitle

\section{Introduction}
In 1954\footnotetext{This work is partially supported by GNAMPA-INdAM}, Gaffney \cite{G}, extended the famous Stokes' theorem to complete
$m$-di\-men\-sio\-nal Riemannian manifolds $M$ by proving that,
given a $C^1$ vector field $X$ on $M$, we have $\int_M\dive X=0$
provided $X\in L^1(M)$ and $\dive X\in L^1(M)$ (but in fact
$(\dive X)_-=\max\{-\dive X, 0\}\in L^1(M)$ is enough).
This result was later extended by Karp \cite{K}, who showed
that the assumption $X\in L^1(M)$ can be weakened to
\[
\liminf_{R\to+\infty}\frac{1}{R}\int_{B_{2R}\setminus B_{R}} |X| dV_M=0.
\]
Here and on, having fixed a reference origin
$o\in M$ on a non-compact manifold $M$, we set $r\left(  x\right)  =\operatorname{dist}_{M} (  x,o)  $ and we
denote by $B_{t}$ and $\partial B_{t}$ the geodesic ball and sphere of radius
$t>0$ centered at $o$. Moreover $dV_M$ is the Riemannian volume measure on $M$.

It turns out that the completeness of a manifold is analogous to the $p=\infty$ case of $p$-parabolicity, i.e., $M$ is complete if and only if it is $\infty$-parabolic. We recall the concept of $p$-parabolicity. The $p$-laplacian of a real valued function $u:M\rightarrow\mathbb{R}$ is defined by $\Delta
_{p}u=\operatorname{div} ( \left\vert \nabla u\right\vert ^{p-2}\nabla
u).$ A function $u\in W^{1,p}_{\rm{loc}}(M)$ is said to be a $p$-subsolution if
$\Delta_{p}u\geq 0$ weakly on $M$. When any bounded above $p$-subsolution is
necessarily constant we say that the manifold $M$ is $p$-parabolic.
A very useful characterization of (non-)$p$-parabolicity goes under the name of the
Kelvin-Nevanlinna-Royden criterion. In the linear setting $p=2$ it was proved
in a paper by Lyons and Sullivan \cite{LS}. See also \cite{PRS-Progress} Theorem 7.27. The non-linear extension, due to Gol'dshtein and Troyanov \cite{GT}, states that a manifold $M$ is not $p$-parabolic if and only if there exists
a vector field $X$ on $M$ such that (a) $\left\vert X\right\vert \in L^{p/(p-1)}\left(  M\right)  $, (b) $\operatorname{div}X\in L_{\text{\rm{loc}}}^{1}\left(  M\right) $ and $\left(  \operatorname{div}X\right)  _{-}\in
L^{1}\left(  M\right)  $ (in the weak sense) and (c) $0<\int_{M}\operatorname{div}X\leq+\infty$. In particular this result shows that if $M$ is $p$-parabolic and $X$ is a vector field on $M$ satisfying (a) and (b), then $\int_M\dive X =0$, thus giving a $p$-parabolic analogue of the Gaffney result. Hence, it is natural to ask whether there exists a $p$-parabolic analogue of the Karp theorem, i.e., if it is possible to weaken the assumptions on the vector field $X$ and still conclude that $\int_M \operatorname{div} X dV_M =0$.

In this paper we will present two different ways to get this
result. The first one, Theorem \ref{teo_ex}, is presented in
Section \ref{sec_evans} and relies on the existence of special
exhaustion functions. It has a more theoretical taste and gives
the desired $p$-parabolic analogue of Karp's theorem, at least in
case either $p=2$ or $p>1$ and $M$ is a model manifold. The second
one, Theorem \ref{th_gt}, is more suitable for explicit
applications. It is presented in Section \ref{volume} and avoids
the parabolicity assumption on $M$ by  requiring some connections
between the $q$-norm of $X$, $q=p/(p-1)$, and the volume growth of
geodesic balls in $M$. In some sense, specified in Remark \ref{cutoff}, this result is optimal. In Section \ref{appl} we use these techniques to generalize some
results involving the $p$-laplacian comparison and uniqueness theorems on the $p$-harmonic representative in a homotopy class. In particular, in
Theorem \ref{th_rn}, we extend a $p$-laplacian comparison for
vector valued maps on $p$-parabolic manifolds recently obtained in \cite{HPV}. We point out that the new proof admits $C^0\cap W^{1,p}_{\rm{loc}}$ maps, instead of the
smooth maps of \cite{HPV}. This is a relevant
improvement since, as opposed to the linear setting where
$2$-harmonic maps are necessarily smooth, for $p\neq 2$ one can at
most ensure $p$-harmonic maps to be $C^{1,\alpha}$
\cite{To},\cite{HL},\cite{WY}. Hence, $C^{1,\alpha}$ seems to be the
natural class of functions to consider when dealing with
the $p$-laplacian.

\section*{Acknowledgments}
We specially thank professors Stefano Pigola, Alberto Setti and Ilkka Holopainen for the useful hints and conversations that helped shape this article. Moreover we would like to thank professor Frank Morgan for suggesting some improvements to the first version of this paper.

\section{Exhaustion functions and parabolicity}\label{sec_evans}

Given a continuous exhaustion function $f:M\to \R^+$ in $W^{1,p}_{\rm{loc}}(M)$, set 
\[
C(r)=f^{-1}[0,2r)\setminus f^{-1}[0,r).
\] 
\begin{definition}
We say that a vector field $X\in L^{q}_{\rm{loc}}(M)$ (with $\frac 1 p + \frac 1 q = 1$) satisfies the condition $\mathcal{E}_{M,p}$ if:
\begin{equation*}%\label{cond_ex}%\tag{$\mathcal{E}_{M,p}$}
 \liminf_{r\to \infty}\frac 1 r \abs{\int_{C(r)} \abs{\nabla{f}}^pdV_M}^{1/p}\abs{\int_{C(r)}\abs X ^{p/(p-1)} dV_M}^{(p-1)/p}=0 .
\end{equation*}
%where $q=p/(p-1)$ is the conjugate exponent of $p$.
\end{definition}

\begin{theorem}\label{teo_ex}
 Let $f:M\to \R^+$ be a continuous exhaustion function in $W^{1,p}_{\rm{loc}}(M)$. If $X$ is a $L^q_{\rm{loc}}(M)$ vector field with $(\dive X)_-\in L^1(M)$, $\dive(X)\in L^1_{\rm{loc}}(M)$ in the weak sense and $X$ satisfies the condition $\mathcal{E}_{M,p}$, then $\int_M \dive(X) dV_M=0$.
\begin{proof}
Note that $(\dive X)_-\in L^1(M)$ and $\dive(X)\in L^1_{\rm{loc}}(M)$ in the weak sense is the most general hypothesis under which $\int_M \dive X dV$ is well-defined (possibly infinite). For $r>0$, consider the $W^{1,p}$ functions defined by
\begin{equation*}
 f_r(x) := \max\{\min\{2r-f(x),r\},0\},
\end{equation*}
i.e., $f_r$ is a function identically equal to $r$ on $D(r):= f^{-1}[0,r)$, with the support in $D(2r)$ and such that $\nabla {f_r}=-\chi_{C(r)}\nabla f $, where $\chi_{C(r)}$ is the characteristic function of $C(r)$. By dominated and monotone convergence we can write
\begin{equation*}
 \int_M \dive X dV_M=\lim_{r\to \infty}\frac 1 r \int_{D(2r)} f_r \dive X dV_M%=\liminf_{r\to \infty}\frac 1 r \int_{D(2r)} f_r \dive(X) dV_M.
\end{equation*}
Since $f$ is an exhaustion function, $f_r$ has compact support, by definition of weak divergence we get
\begin{align*}
\int_M \dive X dV_M &=\lim_{r\to\infty} \frac 1 r\int_{D(2r)}\ps{\nabla f_r}{X} dV_M\\ &\leq\liminf_{r\to \infty} \frac 1 r\ton{\int_{C(r)} \abs{\nabla f}^p dV_M}^{1/p}\ton{\int_{C(r)} \abs{X}^q dV_M}^{1/q}=0.
\end{align*}
This proves that $(\dive X)_+:=\dive X+(\dive X)_-\in L^1(M)$ and by exchanging $X$ with $-X$, the claim follows.
\end{proof}
\end{theorem}
Note that setting $p=\infty$ and $f(x)=r(x)$, one gets exactly the statement of Karp \cite{K}.

\begin{remark}
\rm{From the proof of Theorem \ref{teo_ex}, it is easy to see that condition $\mathcal{E}_{M,p}$ can be generalized a little. In fact, if there is a function (with no regularity assumptions) $g:(0,\infty)\to(0,\infty)$ such that $g(t)>t$ and
\begin{equation*}%\label{cond_exg}\tag{$\mathcal{E}_{M,p}'$}
\liminf_{r\to \infty}\frac 1 {g(r)-r} \abs{\int_{G(r)} \abs{\nabla{f}}^pdV_M}^{1/p}\abs{\int_{G(r)}\abs X ^q dV_M}^{1/q}=0,
\end{equation*}
where $G(r)\equiv f^{-1}[0,g(r))\setminus f^{-1}[0,r]$, the conclusion of Theorem \ref{teo_ex} is still valid with the same proof, only needlessly complicated by an awkward notation.}
\end{remark}

The smaller the value of $\int_{C(r)} \abs{\nabla f}^p dV_M$ is, the more powerful the conclusion of the theorem is, and since $p$-harmonic functions are in some sense minimizers of the $p$-Dirichlet integral, it is natural to look for such functions as candidates for the role of $f$. Of course, if $M$ is $p$-parabolic it does not admit any positive nonconstant $p$-harmonic function defined on all $M$. Anyway, since we are interested only in the behaviour at infinity of functions and vector fields involved (i.e. the behaviour in $C(r)$ for $r$ large enough), it would be enough to have a $p$-harmonic function $f$ which is defined outside a compact set (inside it could be given any value without changing the conclusions of the theorem).

For example, L. Sario and M. Nakai proved that for every $2$-parabolic surface $M$, and every relatively compact set $\Omega$, there exists an Evans' potential (see \cite[theorems 12F and 12G]{SN}), i.e., a positive harmonic exhaustion function $E:M\setminus \Omega \to \R^+$ with $E|_{\partial \Omega}=0$ and such that for any $c>0$
\begin{equation*}
 \int_{\{E(x)\leq c\}} \abs{\nabla E}^2 dV_M \leq c.
\end{equation*}
Putting this into inequality $\mathcal{E}_{M,p}$ (with $p=2$) we can conclude that on a $2$-parabolic surface, a vector field $X\in L^2_{\rm{loc}}(M)$ with $(\dive X)_-\in L^1(M)$ and $\dive(X)\in L^1_{\rm{loc}}(M)$ has zero divergence over $M$ provided
\begin{equation*}%\label{eq_evans}
 \lim_{r\to \infty} \frac{\int_{\{r\leq E(x)\leq 2r\}}\abs X ^2 dV_M} r =0,
\end{equation*}
thus giving a result very similar (at least formally) to Karp's, except for the different exponents and for the presence of the Evans potential $E(\cdot)$ that plays the role of the geodesic distance $r(\cdot)$.

It can be proved that an Evans potential exists not only on ($2$-)parabolic surfaces but also on ($2$-)parabolic manifolds (see \cite{V} for a complete proof). Unfortunately, no similar existence results have been proved yet in the generic non-linear case ($p\neq 2$). Moreover, even for $p=2$, in general there is no explicit characterization of the function $E$ which can help to estimate its level sets, and therefore the quantity $\int_{\{r\leq E(x)\leq 2r\}}\abs X ^q dV_M$. A very special case is given by the model manifolds (in the sense of Greene and Wu \cite{GW}). In this setting the radial function $f:M\setminus B_1\to\mathbb R$ defined as
\begin{gather}\label{func_f}
 f(x):= \int_1^{r(x)} \frac 1 {A (\partial B_s)^{1/(p-1)}} ds
\end{gather}
is $p$-harmonic and it holds
\begin{equation}\label{grad}
\int_{B_{r_2}\setminus B_{r_1}}|\nabla f|^p = \int_{r_1}^{r_2} \frac 1 {A (\partial B_s)^{1/(p-1)}} ds = f(r_2)-f(r_1)\qquad \ \ for r_2>r_1>1.
\end{equation}
Here and in what follows $A\left(  \partial B_{t}\right)  $ stands for the $\left(  m-1\right)  $-dimensional Hausdorff measure of $\partial B_{t}$ and is a.e.\ continuous as a function of $t$. Moreover the model manifold $M$ is $p$-parabolic if and only if $f(\infty)=\infty$ (see \cite{G2} and \cite{T}) and hence in this case $f$ is the Evans' potential we looked for.

\section{Stokes' Theorem under volume growth assumptions}\label{volume}

We now return to consider manifolds $M$ which are not necessarily spherically symmetric. In this general situation, the function $f(r)$ defined in (\ref{func_f}) is not $p$-harmonic. Moreover, since $f(+\infty)=+\infty$ is not in general a necessary condition for $p$-parabolicity (though it is sufficient, see \cite{T}, \cite{RS}, \cite{Holo-Cont}, and Remark \ref{area} below) , we are not ensured that $f$ is an exhaustion function even for $p$-parabolic $M$. However $f$ is still well defined and relation \eqref{grad} still holds. Hence, we are led to generalize Theorem \ref{teo_ex} to generic manifolds, i.e. without parabolicity assumptions, and the generalization we obtain is optimal in some sense; see Theorem \ref{th_gt} and Remark \ref{cutoff} below. Obviously, in this new result the conclusion will depend on the volume growth of geodesics ball of $M$. Key tools are the estimates from above of the capacity of the condenser $(\bar B_{r_1},B_{r_2})$ with surface and volume comparisons, as shown in the next proposition. Before that, we briefly recall the definition of $p$-capacity.

Given a Riemannian manifold $M$, let $\Omega$ be a connected domain in $M$ and $D\subset\Omega$ a compact set. For $p\geq 1$, the $p$-capacity of $D$ in $\Omega$ is defined by
\[
\Capp p(D,\Omega):=\inf\left\{\int_{\Omega} |du|^p\ \vert\ u\in W^{1,p}_0(\Omega) \cap C_0^0(\Omega), u\geq1\textrm{ on }D\right\}.
\]
It is well known that a manifold $M$ is $p$-parabolic if and only if $\Capp p(D,M)=0$ for every compact set $D\subset M$ (or equivalently for a compact set $D$ with nonvoid interior part),\cite{Holo-Dissertation}. For $p>1$, we define the functions $a_p(t),b_p^{(r_1)}:(0,+\infty)\to(0,+\infty)$ as
\[
a_p(t):= A(\partial B_t)^{-1/(p-1)} \ \ \ b_p^{(r_1)}(t):=\ton{\frac{s-r_1}{V(t)-V(r_1)}}^{1/(p-1)}
\]
Since the volume $V(B_t)$ of a geodesic ball seen as a function of $t$ is continuous and differentiable almost everywhere with $V'(B_t)=A(\partial B_t)$ (see for example \cite{Ch} proposition III.3.2), both functions are a.e. continuous in $(0,+\infty)$, and $b_p$ is also differentiable almost everywhere.

In \cite[pag. 3]{G3}, A. Grigor'yan proves the following inequalities for the $p$-capacity of a spherical condenser, inequalities that link the $p$-parabolicity of a manifold to the area and volume growth of its geodesic balls.
\begin{proposition}\label{prop_cfrcap}
 Given a complete Riemannian manifold $M$, the capacity of the condenser $(B_{r_1},B_{r_2})$ is bounded from above by
\begin{equation}\label{eq_cap_s}
 \operatorname{Cap}_p(\bar B_{r_1},B_{r_2})\leq \ton{\int_{r_1}^{r_2} a_p(t) dt}^{1-p},
\end{equation}
\begin{equation}\label{eq_cap_v}
 \operatorname{Cap}_p(\bar B_{r_1},B_{r_2})\leq 2^p\ton{\int_{r_1}^{r_2} b_p^{(r_1)}(t) dt}^{1-p},
\end{equation}
\end{proposition}

\begin{remark}\label{area}
\rm{ We observed before that $a_p\notin L^1(0,+\infty)$ implies $M$ is $p$-parabolic. This is easily obtained by letting $r_2$ go to infinity in \eqref{eq_cap_s}.}
\end{remark}
\begin{remark}
\rm{ In \cite{G2}, the author proves similar inequalities in the case $p=2$, but with a different proof. This proof can be easily adapted to obtain a better constant in inequality \eqref{eq_cap_v}, in fact $2^p$ can be replaced by $p$.}
\end{remark}

%\noindent 
The functions $a_p$ and $b_p$ can be used to construct special cutoff functions with controlled $p$-Dirichlet integral. Using these cutoffs, with an argument similar to the one we used in the proof of Theorem \ref{teo_ex}, we get a more suitable and manageable condition on a vector field $X$ in order to guarantee that $\int_M \dive(X)dV_M=0$.

\begin{definition}
We say that a real function $f:M\to\mathbb R$ satisfies the condition $\mathcal{A}_{M,p}$ on $M$ for some $p>1$ if there exists a function $g:(0,+\infty)\to(0,+\infty)$ such that
\begin{equation}\label{cond_vol}%\tag{$\mathcal{A}_{M,p}$}
\liminf_{R\to\infty}\left(\int_{B_{(R+g(R))}\setminus B_R}fdV_M\right)\left(\int_R^{R+g(G)} a_p(s) ds\right)^{-1}=0.
\end{equation}
\end{definition}

The next result gives the annunced generalization under volume growth assumption of the Kelvin-Nevanlinna-Royden criterion.

\begin{theorem}\label{th_gt}
Let $(M,\left\langle ,\right\rangle)$ be a non-compact Riemannian manifold. Let $X$ be a vector field on $M$ such that
\begin{equation}\label{liploc}
\operatorname{div}X\in L_{\text{\rm{loc}}}^{1}\left(  M\right)\ \mathrm{and}\ \max\left(
-\operatorname{div}X,0\right)  =\left(  \operatorname{div}X\right)  _{-}\in
L^{1}\left(  M\right).
\end{equation}
If $|X|^{p/(p-1)}$ satisfies the condition $\mathcal{A}_{M,p}$ on $M$, then
\[
\int_M\operatorname{div}X dV_M=0.
\]
%ere exists a function $g:(0,+\infty)\to(0,+\infty)$ such that
%\begin{equation}\label{liminf}
%\liminf_{R\to\infty}\left(\int_{B_{(R+g(R))}\setminus B_R}{|X|^{p/(p-1)}dV_M}\right)\left(\int_R^{R+g(G)} a_p(s) ds\right)^{-1}=0.
%\end{equation}
\end{theorem}

Accordingly, if $X$ is a vector field on $M$ such that $\abs X ^{p/(p-1)}$ satisfies the condition $\mathcal{A}_{M,p}$, $\operatorname{div}X\in L_{\text{\rm{loc}}}^{1}\left(  M\right)  $, and $\operatorname{div}X\geq0$ on $M$, then we must necessarily conclude that
$\operatorname{div}X=0$ on $M$. As a matter of fact, even if $\operatorname{div}X\notin L_{\text{\rm{loc}}}^{1}\left(  M\right)  $, we can obtain a similar conclusion as shown in the next proposition, inspired by \cite[Proposition 9]{HPV}.

\begin{proposition}\label{weak_KNR}
Let $(M,\left\langle ,\right\rangle)$ be a non-compact Riemannian manifold. Let $X$ be a vector field on $M$ such that
\begin{equation}\label{div_deb}
\operatorname{div}X\geq f
\end{equation}
in the sense of distributions, for some $f\in L^1_{\text{\rm{loc}}}(M)$ with $f_-\in L^1(M)$. If $|X|^{p/(p-1)}$ satisfies condition $\mathcal{A}_{M,p}$ on $M$ for some $p>1$, then
\begin{equation}\label{int_f}
\int_M f \leq 0.
\end{equation}
\end{proposition}

\begin{remark}\label{parabolic_weak_KNR}\rm{
Combining the following proof with the proof of \cite[Proposition 9 and
Remark 10]{HPV}, one obtains the validity of \eqref{int_f}
when $M$ is $p$-parabolic and $|X|^{p/(p-1)}\in L^1(M)$
instead of satisfying $\mathcal{A}_{M,p}$.
}\end{remark}

\begin{proof}
Fix $r_2>r_1>0$ to be chosen later. Define the functions $\hat\varphi=\hat\varphi_{r_1,r_2}:B_{r_2}\setminus B_{r_1}\to\mathbb R$ as
\begin{equation}\label{a1}
\hat\varphi(x):=\left(\int_{r_1}^{r_2}a_p(s)ds\right)^{-1}\int_{r(x)}^{r_2}a_p(s)ds
\end{equation}
and let $\varphi=\varphi_{r_1,r_2}:M\to\mathbb R$ be defined as
\begin{equation}\label{a2}
\varphi(x):=\begin{cases}
1 & r(x)<r_1,\\
\hat\varphi(x) & r_1\leq r(x) \leq r_2,\\
0 & r_2<r(x).
%\\0 & r(x)>r_2.
\end{cases}
\end{equation}
A straightforward calculation yields
\[
\int_M|\nabla\varphi|^p dV_M=\int_{B_{r_2}\setminus B_{r_1}}|\nabla\hat\varphi|^p dV_M = \left(\int_{r_1}^{r_2}a_p(s)ds\right)^{1-p}.
\]
By standard density results we can use $\varphi\in W^{1,p}_0(M)$ as a test function in the weak relation \eqref{div_deb}, thus obtaining
\begin{align}\label{weak_ineq}
\int_{M}\varphi f dV_M &\leq (\operatorname{div}X,\varphi )\\
&=-\int_{M}\left\langle X,\nabla\varphi\right\rangle\nonumber\\
&\leq \left(\int_{\operatorname{supp}\left(\nabla\varphi\right)}|X|^{p/(p-1)}\right)^{(p-1)/p}\left(\int_M|\nabla\varphi|^p\right)^{1/p}\nonumber\\
&\leq \left\{\left(\int_{B_{r_2}\setminus B_{r_1}}|X|^{p/(p-1)}\right)\left(\int_{r_1}^{r_2}a_p(s)ds\right)^{-1}\right\}^{(p-1)/p}, \nonumber
\end{align}
where we have applied H\"older in the next-to-last inequality. Now, let $\left\{R_k\right\}_{k=1}^{\infty}$ be a sequence such that $R_k\to\infty$, which realizes the $\liminf$ in condition \eqref{cond_vol}. Up to passing to a subsequence, we can suppose $R_{k+1}\geq R_k+g(R_k)$. Hence, the sequence of cut-offs $\varphi_k:=\varphi_{R_k,R_k+g(R_k)}$ converges monotonically to $1$ and applying monotone and dominated convergence, we have
\begin{align*}
\lim_{k\to\infty}\int_M \varphi_kf
= \lim_{k\to\infty}\int_M \varphi_k f_+ -\lim_{k\to\infty}\int_M\varphi_k f_- =\int_M f_+-\int_M f_- =\int_Mf.
\end{align*}
Taking limits as $k\to\infty$ in inequality \eqref{weak_ineq}, assumption $\mathcal{A}_{M,p}$ finally gives
\[
\int_{M}f\leq 0.
\]
\end{proof}

\begin{proof}[\textsc{Proof of Theorem} \ref{th_gt}]
Choosing $f=\dive X$ in Proposition \ref{weak_KNR} we get $\int_M\dive X\leq 0$ and $(\dive X)_+\in L^1(M)$. Hence, we can repeat the proof replacing $X$ with $-X$.
%Fix $0<r_1<r_2$. Consider the vector field $\tilde X:=\varphi X$, where the compactly supported functions $\varphi=\varphi_{r_1,r_2}\in C^{0}_c(M)\cap W^{1,p}_0$ are defined as in the proof of Theorem \ref{weak_KNR}. Applying the weak divergence theorem to $\tilde X$ and computing as in (\ref{weak_ineq}), we deduce
%\[
%\int_M\varphi\dive X dV_M\leq \left\{\left(\int_{B_{r_2}\setminus B_{r_1}}|X|^{p/(p-1)}dV_M\right)\left(\int_{r_1}^{r_2}a_p(s)ds\right)^{-1}\right\}^{\frac{p-1}{p}}.
%\]
%The thesis follows reasoning as in the proofs of Theorem \ref{teo_ex} and Proposition \ref{weak_KNR}.
%On the other hand, since $\left(  \operatorname{div}X\right)_{-}\in L^{1}\left(  M\right)$, we can write
%\[
%\int_M\varphi\dive X dV_M \geq \int_{B_{r_1}}\left(\dive X\right)_+ dV_M - \int_M \left(  \operatorname{div}X\right)_{-}dV_M.
%\]
%Choosing a sequence $\left\{R_k\right\}_{k=1}^{\infty}$ as in the proof of Theorem \ref{weak_KNR} and letting $k\to\infty$, condition $\mathcal{A}_{M,p}$ yields $\left(  \operatorname{div}X\right)_{+}\in L^1(M)$ and $\int_M\dive X dV_M\leq 0$.\\
%Finally we can repeat all the reasonement with $-X$ instead of $X$, thus concluding the proof.
\end{proof}

\begin{remark}\label{cutoff}\rm{
We point out that one could easily obtain results similar to Theorem \ref{th_gt} and Proposition \ref{weak_KNR} replacing in the proofs $\varphi_{r_1,r_2}$ with standard cut-off functions $0\leq\xi_{r_1,r_2}\leq1$ defined for any $\varepsilon>0$ in such a way that
\[
\xi_{r_1,r_2}\equiv 1\textrm{ on }B_{r_1},\qquad \xi_{r_1,r_2}\equiv 0\textrm{ on }M\setminus B_{r_2},\qquad |\nabla\xi_{r_1,r_2}|\leq\frac {1+\varepsilon} {r_2-r_1}.
\]
%Nevertheless, the functions $\varphi_{r_1,r_2}$ introduced here are optimal, in the sense that, at least on model manifolds, they are $p$-harmonic, thus minimizing the $p$-energy.
Nevertheless $\varphi_{r_1,r_2}$ gives better results than the standard cutoffs. For example, consider a $2$-dimensional model manifold with the Riemannian metric
\begin{gather*}
 ds^2 = dt^2 + g^2(t) d\theta^2,
\end{gather*}
where $g(t)=e^{-t}$ outside a neighborhood of $0$. The $p$-energy of $\varphi_{r_1,r_2}$ and $\xi_{r_1,r_2}$ are respectively
\begin{gather*}
 \int_M \abs{\nabla \varphi_{r_1,r_2}}^p dv_M = \left(\int_{r_1}^{r_2}a_p(s)ds\right)^{1-p}= 2\pi (p-1)^{1-p}\ton{e^{{r_2}/(p-1)}-e^{{r_1}/({p-1})}}^{1-p},\\
\int_M \abs{\nabla \xi_{r_1,r_2}}^p dv_M \leq \ton{\frac{1+\epsilon}{r_2-r_1}}^p\int_{r_1}^{r_2} A(\partial B_s) ds =2\pi\ton{\frac{1+\epsilon}{r_2-r_1}}^p \ton{e^{-r_1}-e^{-r_2}}.
\end{gather*}
If we choose $r_2=2r_1\equiv 2r$ and let $r\to \infty$, we get
\begin{gather*}
  \int_M \abs{\nabla \varphi_{r,2r}}^p dv_M\sim c e^{-r} \ton{e^{{r}/({p-1})}-1}^{1-p}\sim c e^{-2r}\\
\int_M \abs{\nabla \xi_{r,2r}}^p dv_M \sim \frac{c'}{r^p} e^{-r}\ton{1-e^{-r}}\sim \frac{c'}{r^p}e^{-r},
\end{gather*}
for some positive constants $c,c'$. Using $\varphi_{r,2r}$ in the proof of Proposition \ref{weak_KNR} (in particular in inequality \eqref{weak_ineq}), we can conclude that $f=0$ provided
\begin{align*}
\lim_{r\to \infty} \left(\int_{B_{2r}\setminus B_r}|X|^{p/(p-1)}\right)^{(p-1)/{p}} c e^{-2r}=0,
\end{align*}
while using $\xi_{r,2r}$ we get a weaker result, i.e., $f=0$ provided
\begin{align*}
\lim_{r\to \infty} \left(\int_{B_{2r}\setminus B_r}|X|^{p/(p-1)}\right)^{{(p-1)}/{p}} \frac{c}{r^p}e^{-r}=0.
\end{align*}

One could ask if there exist even better cutoffs than the ones we chose. First, note that restricting to model manifolds, the cutoffs $\varphi_{r_1,r_2}$ are $p$-harmonic on $B_{r_2}\setminus B_{r_1}$, and so their $p$-energy is minimal. In general this is not true. Anyway, it turns out that the functions $\varphi_{r_1,r_2}$ are optimal at least in the class of radial functions, in fact they minimize the $p$-energy in this class, making condition $\mathcal{A}_{M,p}$ sharp. To prove this fact, consider any radial cutoff $\psi:=\psi_{r_1,r_2}$ satisfying $\psi\equiv 1\textrm{ on }B_{r_1}, \ \psi\equiv 0\textrm{ on }M\setminus B_{r_2}$. By Jensen's inequality we have
\begin{gather*}
 \int_{} \abs{\nabla\varphi_{r_1,r_2}}^p dv_M = \left(\int_{r_1}^{r_2}a_p(s)ds\right)^{1-p}=c_\psi^{1-p}\left(\int_{r_1}^{r_2}\frac{a_p(s)}{\abs {\psi'(s)}}\frac{\abs {\psi'(s)}ds}{c_\psi}\right)^{1-p} \\
\leq c_\psi^{-p}\int_{r_1}^{r_2} \abs{\psi'(s)}^p A(\partial B_s) ds \leq \int_{} \abs{\nabla\psi}^p dv_M,
\end{gather*}
where $\psi'$ is the radial derivative of $\psi$ and $c_\psi=\int_{r_1}^{r_2} \abs{\psi'(s)}ds\geq 1$.
}\end{remark}

\section{Applications}\label{appl}

Theorem \ref{th_gt}, and Proposition \ref{weak_KNR}, can be naturally applied to those situations where the standard Kelvin-Nevanlinna-Royden criterion is used to deduce information on $p$-parabolic manifolds. First, we present a global comparison result for the $p$-laplacian of real valued functions. The original result assuming $p$-parabolicity appears in \cite[Theorem 1]{HPV}.

\begin{theorem}
\label{th_comparison}Let $\left(  M,\left\langle ,\right\rangle \right)  $ be
a connected, non-compact Riemannian manifold. Assume that
$u,v\in W^{1,p}_{\text{\rm{loc}}}(M)\cap C^0(M)$, $p>1$, satisfy%
\[
\Delta_{p}u\geq\Delta_{p}v\text{ weakly on }M,
\]
and that $|\nabla u|^p$ and $|\nabla v|^p$ satisfy condition $\mathcal{A}_{M,p}$ on $M$. Then, $u=v+A$ on $M$, for some constant $A\in\mathbb{R}$.
\end{theorem}

Choosing a constant function $v$, we immediately deduce the following result, which generalize \cite[Corollary 3]{PRS-MathZ}.

\begin{corollary}\label{subharm}
Let $\left(  M,\left\langle ,\right\rangle \right)  $ be
a connected, non-compact Riemannian manifold. Assume that
$u\in W^{1,p}_{\text{\rm{loc}}}(M)\cap C^0(M)$, $p>1$, is a weak p-subharmonic function on $M$ such that $|\nabla u|^p$ satisfies condition $\mathcal{A}_{M,p}$ on $M$. Then $u$ is constant.
\end{corollary}

In \cite{SY}, R. Schoen and S.T. Yau considered the problem of uniqueness of the 2-harmonic representative with finite energy in a (free) homotopy class of maps from a complete manifold $M$ of finite volume to a complete manifold $N$ of non-positive sectional curvature. In particular, they obtained that if ${}^N\operatorname{Sect}<0$ then a given harmonic map $u$ is unique in its homotopy class, unless $u(M)$ is contained in a geodesic of $N$. Moreover, if the sectional curvature of $N$ is nonpositive and two homotopic harmonic maps $u$ and $v$ with finite energy are given, then $u$ and $v$ are homotopic through harmonic maps. In \cite{PRS-MathZ}, Pigola, Setti and Rigoli noticed that the assumption $V(M)<\infty$ can be replaced by asking $M$ is $2$-parabolic. In Schoen and Yau's result, the finite energy of the maps is used in two fundamental steps of the proof:
\begin{enumerate}
 \item to prove that a particular subharmonic map of finite energy is constant;
 \item to construct the homotopy via harmonic maps.
\end{enumerate}
Using the $p=2$ case of Corollary \ref{subharm}, we can deal with step (1) and thus obtaine the following theorem. If ${}^N\operatorname{Sect}\leq0$, weakening finite energy assumption in step (2) does not seem trivial to us, but we can still get a result for maps with fast $p$-energy decay without parabolicity assumption.

\begin{theorem}\label{SY}
Suppose $M$ and $N$ are complete manifolds.
\begin{enumerate}
\item[1)] Suppose ${}^N\operatorname{Sect}<0$. Let $u:M\to N$ be a harmonic map such that $|\nabla u|^2$ satisfies condition $\mathcal{A}_{M,2}$ on $M$. Then there is no other harmonic map homotopic to $u$ satisfying condition $\mathcal{A}_{M,2}$ unless $u(M)$ is contained in a geodesic of $N$.
\item[2)] Suppose ${}^N\operatorname{Sect}\leq0$. Let $u,v:M\to N$ be homotopic harmonic maps such that $|\nabla u|^2,|\nabla v|^2\in L^1(M)$ satisfy condition $\mathcal{A}_{M,2}$ on $M$. Then there is a smooth one parameter family $u_t:M\to N$ for $t\in [0,1]$ of harmonic maps with $u_0=u$ and $u_1=v$. Moreover, for each $x\in M$, the curve $\{u_t(x):t\in [0,1]\}$ is a constant $($independent of $x)$ speed parametrization of a geodesic.
\end{enumerate}
\end{theorem}

\begin{remark}\rm{
While the existence in a homotopy class of a harmonic representative with finite energy is ensured by a further result by Schoen and Yau \cite{SY-CH}, in the setting of Theorem \ref{SY} we are not able to guarantee that there exists at least one harmonic map whose energy satisfies $\mathcal{A}_{M,2}$.
}\end{remark}

An interesting task is to extend Schoen and Yau's uniqueness results to the nonlinear ($p\neq 2$) setting. In \cite{PRS-MathZ}, the authors take a first step in this direction by proving that a map $u:M\to N$ with finite $p$-energy and homotopic to a constant is constant provided $M$ is $p$-parabolic and ${}^N\operatorname{Sect}\leq0$. Using Theorem \ref{th_gt} in the proof of their result, we easily obtain the following

\begin{theorem}Let $(M,\left\langle ,\right\rangle_M)$ and $(N,\left\langle ,\right\rangle_N)$ be complete Riemannian manifolds. Assume that
$M$ is non-compact and that $N$ has non-positive sectional curvatures. If $u :M\to N$ is a $p$-harmonic map homotopic to a constant and with energy density $|du|^p$ satisfying condition $\mathcal{A}_{M,p}$, then $u$ is a constant map.
\end{theorem}

In \cite{HPV}, the authors apply the Kelvin-Nevanlinna-Royden
criterion to obtain a vector valued version of their
comparison theorem. In some sense this result is a further step in
treating the problem of the uniqueness of $p$-harmonic
representative. In particular, if $M$ is $p$-parabolic and
$u,v:M\to\mathbb R^n$ are $C^{\infty}$ maps satisfying $\Delta_p u =
\Delta_p v$ and $|du|,|dv|\in L^p$, then $u=v+A$ for some constant
vector $A\in\mathbb R^n$. To prove it, they construct a vector
field $X$ depending on $du$ and $dv$, whose divergence is such
that
\begin{equation}\label{div_neg}
0\neq (\dive X)_-\leq C(|du|^p+|dv|^p)\in L^1.
\end{equation}
As a matter of fact, in their proof $X$ is defined in such a way
that the smoothness of $u$ and $v$ seems to be strictly necessary to do
computations. In order to generalize their result, in the direction of Proposition
\ref{weak_KNR}, apparently assumption (\ref{div_neg}) can not be
dropped. Nevertheless, in case $|du|,|dv|\in L^p$ and their
$L^p$-norms decay fast with respect to the volume in the
$\mathcal{A}_{M,p}$ sense, we obtain a similar result for non
$p$-parabolic manifolds and for maps with low regularity.

\begin{theorem}\label{th_rn}
Suppose that $\left(  M,\left\langle ,\right\rangle \right)  $ is
a complete non-compact Riemannian manifold. For $p>1$, let
$u,v:M\rightarrow\mathbb{R}^{n}$ be $C^0\cap W^{1,p}_{\rm{loc}}(M)$
maps satisfying%
\begin{equation}
\Delta_{p}u=\Delta_{p}v\text{ on }M, \label{assumption_pDelta}%
\end{equation}
in the sense of distributions on $M$ and%
\[
\left\vert du\right\vert ,\left\vert dv\right\vert \in L^{p}\left(
M\right) .
\]
Suppose either $M$ is $p$-parabolic or $|du|^p$ and $|dv|^p$ satisfy
condition $\mathcal{A}_{M,p}$ on $M$. Then $\ u=v+C\mathrm{,}$ for some
constant $C\in\mathbb{R}^{n}$.
\end{theorem}

\begin{remark}\rm{
In Proposition \ref{prop_cfrcap}, we saw that the capacity of a condenser \linebreak $(B_{r_1},B_{r_2})$ can be estimated from above using either the behaviour of $V(B_s)$ or the behaviour of $A(\partial B_s)$. This suggests that condition $\mathcal{A}_{M,p}$ should have an analogue in which $A(\partial B_s)$ is replaced by $V(B_s)$. This fact is useful since it is usually easier to verify and to handle volume growth assumptions than area growth conditions.}
\begin{definition}
We say that a real function $f:M\to\mathbb R$ satisfies condition \ref{cond_vol2} on $M$ for some $p>1$ if there exists a function $g:(0,+\infty)\to(0,+\infty)$ such that
\begin{equation}\label{cond_vol2}\tag{$\mathcal{V}_{M,p}$}
\liminf_{R\to\infty}\left(\int_{B_{(R+g(R))}\setminus B_R}fdV_M\right)\left(\int_R^{R+g(R)} \ton{\frac{t}{V(t)}}^{1 /(p-1)} ds\right)^{-1}=0.
\end{equation}
\end{definition}
\rm{Indeed, it turns out that in every proposition stated in Section \ref{appl}, condition $\mathcal{A}_{M,p}$ can be replaced by condition \ref{cond_vol2}, and the proofs remain almost the same.}
\end{remark}

\section{Proof of Theorem \ref{th_rn}}

We can now proceed to prove Theorem \ref{th_rn}. Recall that, by
definition, \eqref{assumption_pDelta} holds in the sense of
distributions on $M$ if
\[
\int_M \langle\eta,|dv|^{p-2}dv-|du|^{p-2}du\rangle_{HS}=0
\]
for every compactly supported $\eta\in T^{\ast}M\otimes\mathbb
R^n$. Here $\langle,\rangle_{HS}$ stands for the Hilbert-Schmidt
scalar product on $T^{\ast}M\otimes\mathbb R^n$. Moreover, we
mention the following lemma, derived by a basic inequality by
Lindqvist \cite{L}, which will be useful later.

\begin{lemma}[\cite{HPV} Corollary 5]
\label{coro_lind} Let $\left(  V,\left\langle ,\right\rangle
\right)  $ be a finite dimensional,
real vector space endowed with a positive definite scalar product and let $p>1$. Then, for every $x,y\in V$, it holds%
\begin{equation}
\left\langle \left\vert x\right\vert ^{p-2}x-\left\vert
y\right\vert
^{p-2}y,x-y\right\rangle \geq 2C(p)\Psi(x,y), \label{MHCK}%
\end{equation}
where
\[
\Psi(x,y):=\begin{cases}\left\vert
x-y\right\vert ^{p} & p\geq2\\
\left\vert x-y\right\vert ^{2}/(|x|+|y|)^{2-p}&
1<p<2,\end{cases}
\]
and $C(p)$ is a positive constant depending only on $p$.
\end{lemma}

\begin{proof}[\textsc{Proof of Theorem \ref{th_rn}}]

First, we assume that $M$ is
$p$-parabolic. We suppose that at least one of $u$ or $v$ is non-constant,
for otherwise there is
nothing to prove. Fix $q_{0}\in M$ and set $C:=u(q_{0})-v(q_{0}%
)\in\mathbb{R}^{n}$. Replacing $v$ with $\tilde v := v+ C$ if necessary, we can suppose $C=0$. Introduce the radial function $r:\mathbb{R}%
^{n}\rightarrow\mathbb{R}$ defined as $r(x)=|x|$. For $A>1$,
consider the weakly differentiable vector field $X_{A}$ defined as
\[
X_{A}(x):=\left[  \left.  dh_{A}\right\vert _{(u-v)(x)}\circ\left(  |du(x)|^{p-2}%
du(x)-|dv(x)|^{p-2}dv(x)\right)  \right]  ^{\sharp}%
,\quad x\in M,
\]
where $h_{A}\in C^{\infty}(\mathbb{R}^{n},\mathbb{R})$ is the
function
\[
h_{A}(y):=\sqrt{A+r^2(y)}
\]
and $\sharp$ denotes the usual musical isomorphism defined by $\langle\omega^{\sharp},V\rangle=\omega(V)$
for all differential 1-forms $\omega$ and vector fields $V$. We
observe that $X_{A}$ is well defined since there exists a
canonical identification
\[
T_{(u-v)(q)}\mathbb{R}^{n}\cong T_{u(q)}\mathbb{R}^{n}\cong T_{v(q)}%
\mathbb{R}^{n}\cong\mathbb{R}^{n}.
\]
Compute
%$\operatorname{Hess}(r)=r^{-1}(\left\langle
%,\right\rangle _{\mathbb{R}^{n}}-dr\otimes dr)$ on $\mathbb{R}^{n}%
%\setminus\left\{  0\right\}  $. Thus,
%\[
%\operatorname{Hess}(\frac{r^2}{2})=dr\otimes dr+r\operatorname{Hess}(r)=\left\langle ,\right\rangle
%_{\mathbb{R}^{n}},
%\]
%and
\begin{align*}
&dh_A=\frac{dr^2}{2\sqrt{A+r^2}}\\
%&\operatorname{Hess}h_A = \frac{\operatorname{Hess}\frac{r^2}{2}}{\sqrt{A+r^2}}-\frac{d\frac{r^2}{2}\otimes d\frac{r^2}{2}}{(A+r^2)^{3/2}} = \frac{\left\langle, \right\rangle_{\mathbb R^n}}{\sqrt{A+r^2}}-\frac{r^2}{(A+r^2)^{3/2}}dr\otimes dr.
\end{align*}
and observe that, because of the special structure of $\mathbb R^n$, for
each vector field $Y$ on $\mathbb R^n$ it holds
\[
(dr^2)|_{(u-v)(x)}(Y) = 2\left\langle
(u-v)(x),Y\right\rangle_{\mathbb R^n} .
\]
By definition of weak
divergence, for each test function $0\leq\phi\in C^{\infty}_c(M)$, we have
\begin{align*}%\label{weak_div}
&-(\dive X_A,\phi)=\int_M \left\langle X_A,\nabla\phi\right\rangle_M \\
&= \int_M \left\langle \left[  \left.  dh_{A}\right\vert _{(u-v)(x)}\circ\left(  |du(x)|^{p-2}%
du(x)-|dv(x)|^{p-2}dv(x)\right)  \right]  ^{\sharp},\nabla\phi\right\rangle_M\nonumber\\
&=\int_M \left.  \frac{dr^2}{2\sqrt{A+r^2}}\right\vert _{(u-v)(x)}\circ\left(  |du(x)|^{p-2}%
du|_x-|dv(x)|^{p-2}dv|_x\right) (\nabla\phi)\nonumber\\
&=\int_M \frac{1}{\sqrt{A+r^2(u-v)(x)}}\left\langle (u-v)(x),\left(|du(x)|^{p-2}%
du|_x-|dv(x)|^{p-2}dv|_x\right)(\nabla\phi(x))\right\rangle_{\mathbb
R^n}.\nonumber
\end{align*}
Since $u,v\in W^{1,p}_{\rm{loc}}(M)$, assumption
\eqref{assumption_pDelta} implies that

\begin{align}\label{long}
0&= \int_M \left\langle d\left(\frac{(u-v)\phi}{\sqrt{A+r^2(u-v)}}\right),|du|^{p-2}%
du-|dv|^{p-2}dv\right\rangle_{HS}\\
&=\int_M \frac{1}{\sqrt{A+r^2(u-v)}}\left\langle d\phi \otimes (u-v),|du|^{p-2}%
du-|dv|^{p-2}dv\right\rangle_{HS}\nonumber\\
&+\int_M \frac{\phi}{\sqrt{A+r^2(u-v)}}\left\langle du-dv,|du|^{p-2}%
du-|dv|^{p-2}dv\right\rangle_{HS}\nonumber\\
&-\int_M \frac{\phi }{2\left(A+r^2(u-v)\right)^{3/2}}\left\langle dr^2|_{(u-v)}\circ(du-dv) \otimes (u-v),|du|^{p-2}%
du-|dv|^{p-2}dv\right\rangle_{HS}\nonumber
\end{align}
\begin{align*}
\geq& -(\dive X_A,\phi)\nonumber\\
&+\int_M \frac{2C(p)\phi}{\sqrt{A+r^2(u-v)}}\Psi(du,dv)\nonumber\\
&-\int_M\frac{\phi r^2(u-v)}{(A+r^2(u-v))^{3/2}}(|du|+|dv|)(|du|^{p-1}%
+|dv|^{p-1})\nonumber,
\end{align*}
where we have used Lemma \ref{coro_lind} for the second term and
Cauchy-Schwarz inequality for the
third one. Setting
\[
f_A:=\frac{2C(p)}{\sqrt{A+r^2(u-v)}}\Psi -\frac{2
r^2(u-v)}{(A+r^2(u-v))^{3/2}}(|du|^p+|dv|^p),
\]
by Young's inequality, \eqref{long} gives
\begin{equation}\label{w_div}
\dive X_A\geq f_A
\end{equation}
in the sense of distributions.

Let us now compute the $L^{p/(p-1)}%
$-norm of $X_{A}.$ Since
\[
\left\vert |du|^{p-2}du-|dv|^{p-2}dv\right\vert
^{p/(p-1)}\leq\left( |du|^{p-1}+|dv|^{p-1}\right)
^{p/(p-1)}\leq2^{\frac{1}{p-1}}\left(
|du|^{p}+|dv|^{p}\right)  ,
\]
we have
\begin{align*}
|X_A|^{p/(p-1)}&=\left|\sqrt{\frac{r^2(u-v)}{A+r^2(u-v)}}\right|^{p/(p-1)}\left\vert |du|^{p-2}du-|dv|^{p-2}dv\right\vert ^{p/(p-1)}\\
&\leq 2^{\frac{1}{p-1}}\left( |du|^{p}+|dv|^{p}\right) \in L^1(M)
.
\end{align*}
Hence $X_A$ is a weakly differentiable vector field with
$|X_{A}|\in L^{p/(p-1)}(M)$.

To apply Proposition
\ref{weak_KNR} (in the $p$-parabolic version pointed out in Remark
\ref{parabolic_weak_KNR}), it remains to show that $(f_A)_-\in
L^{1}(M)$. To this purpose, we note that

\begin{align}\label{div_-}
(f_A)_-&
\leq \frac{2r^2(u-v)}{(A+r^2(u-v))^{3/2}}(|du|^{p}%
+|dv|^{p})\\
&\leq \frac{r^2(u-v)}{A+r^2(u-v)}\frac{2}{\sqrt{A+r^2(u-v)}}(|du|^{p}%
+|dv|^{p})\nonumber\\
&\leq \frac 2 {\sqrt{A}}(|du|^{p}%
+|dv|^{p})\in L^1(M)\nonumber.
\end{align}
%and
%
%\begin{align}\label{div_--}
%\frac{2C(p)}{\sqrt{A+r^2(u-v)}}|du-dv|^{p}
%\leq \frac {2^p C(p)} {\sqrt{A}}(|du|^{p}%
%+|dv|^{p})\in L^1(M).
%\end{align}
Then, the assumptions of Proposition \ref{weak_KNR} are satisfied
and we get, for every $A>1$,

\begin{align}\label{div0}
0&\geq\int_{M}f_{A}\\
&=\int_{M}\left[\frac{2C(p)}{\sqrt{A+r^2(u-v)}}\Psi
-\frac{2r^2(u-v)}{(A+r^2(u-v))^{3/2}}(|du|^{p}%
+|dv|^{p})\right].\nonumber
\end{align}
Fix $T>0$ and define

\[
M_T:=\left\{x\in M \ \vert\  r(u-v)(x)\leq T\right\}\quad\textrm{and}\quad M^T:= M\setminus M_T.
\]
Then, we can write (\ref{div0}) as

\begin{align}\label{div01}
0&\geq\int_{M^T}f_A + \int_{M_T}\frac{2C(p)}{\sqrt{A+r^2(u-v)}}\Psi
-\int_{M_T}\frac{2r^2(u-v)}{(A+r^2(u-v))^{3/2}}(|du|^{p}%
+|dv|^{p}).%\nonumber
\end{align}
Note that

\begin{align}\label{div1}
&\int_{M^T}f_A\geq
-\int_{M^T}\frac{2}{\sqrt{A+T^2}}(|du|^{p}%
+|dv|^{p}) = -\frac{2}{\sqrt{A+T^2}}\int_{M^T}(|du|^{p}%
+|dv|^{p}).
\end{align}

On the other hand, to deal with $\int_{M_T}f_A$, observe that

\begin{align}\label{div2}
\int_{M_T}\frac{2C(p)}{\sqrt{A+r^2(u-v)}}\Psi \geq
\frac{2C(p)}{\sqrt{A+T^2}}\int_{M_T}\Psi.
\end{align}
Furthermore, the real function $t\mapsto2t/(A+t)^{3/2}$ has a
global maximum at $t=2A$ and is increasing in $(0,2A)$. Hence, up
to choosing $A>T^2/2$, we have also

\begin{align}\label{div3}
\int_{M_T}\frac{2r^2(u-v)}{(A+r^2(u-v))^{3/2}}(|du|^{p}%
+|dv|^{p})
\leq \frac{2T^2}{(A+T^2)^{3/2}}\int_{M_T}(|du|^{p}%
+|dv|^{p}).
\end{align}
Inserting \eqref{div1}, \eqref{div2} and \eqref{div3} in
\eqref{div01}, we get

\begin{align*}
\frac{2C(p)}{\sqrt{A+T^2}}\int_{M_T}\Psi
&\leq \frac{2T^2}{(A+T^2)^{3/2}}\int_{M_T}(|du|^{p}%
+|dv|^{p})\\
&+ \frac{2}{\sqrt{A+T^2}}\int_{M^T}(|du|^{p}%
+|dv|^{p}),
\end{align*}
which gives

\begin{align*}
C(p)\int_{M_T}\Psi
\leq \int_{M^T}(|du|^{p}%
+|dv|^{p})+ \frac{T^2}{{\sqrt{A+T^2}}}\int_{M_T}(|du|^{p}%
+|dv|^{p}),
\end{align*}
for all $A>\max\{1,T^2/2\}$. Letting $A\to+\infty$, this latter
yields

\begin{align}\label{fin}
C(p)\int_{M_T}\Psi
\leq\int_{M^T}(|du|^{p}%
+|dv|^{p}).
\end{align}

Since $(|du|^p+|dv|^p)\in L^1(M)$, for $T\to\infty$ we can apply respectively dominated convergence on the right-hand side and monotone convergence on the left-hand side of (\ref{fin}), thus getting

\[
C(p)\int_{M}\Psi=0,
\]
which in turn gives $|d(u-v)|\equiv 0$ on $M$, that is, $u-v\equiv
u(q_{0})-v(q_{0})=C$ on $M$. This conclude the first part of the
proof.

If $M$ is not $p$-parabolic, proceeding as above, condition
$\mathcal{A}_{M,p}$ and Proposition
\ref{weak_KNR} give \eqref{div0}. From there on, we can repeat the
proof of the $p$-parabolic case.
\end{proof}

\section{Final remarks}
So far, we applied Theorem \ref{th_gt} and Proposition \ref{weak_KNR} to generalize the results for which originally a $p$-parabolicity assumption on the domain manifold was used to apply the Kelvin-Nevanlinna-Royden criterion and deduce a Stokes' type conclusion. In particular we showed that the $p$-parabolicity can be replaced by suitable volume growth estimates. As a matter of fact the functions $\varphi_{r_1,r_2}$ defined in the proof of Proposition \ref{prop_cfrcap} seem to naturally appear each time a control on the $L^p$-norm of the gradient of a cut-off function is required. For instance, consider a $m$-dimensional manifold $M$ supporting an Euclidean type Sobolev inequality, i.e.,
\begin{equation}\label{sob}
\left\|\eta\right\|_p^q\leq S_M^q\left\|\nabla\eta\right\|_q^q\qquad \text{for all } \eta\in C_c^{\infty}(M)
\end{equation}
holds for some positve constant $S_M$ and for some $1<q<m$ and $p=mq/(m-q)$. For such manifolds, Carron \cite{C}, proved that there is an almost Euclidean lower bound for the volume growth of the geodesic ball. Namely, there exists an explicit positive $\gamma>0$ depending on $m$ and $q$ such that
\begin{equation}\label{carron}
V(B_r)\geq \gamma r^m,
\end{equation}
for all $r>0$ (but to the best of our knowledge no lower control on $A(\partial B_r)$ is given). Moreover Cao, Shen and Zhu \cite{CSZ} and Li and Wang \cite{LW} observed that the validity of (\ref{sob}) for $q=2$ implies $M$ is $2$-hyperbolic (see also \cite{PST} for the $q\neq2$ case). Observe that, by a standard density argument, inequality (\ref{sob}) holds for all $\eta\in W_0^{1,q}(M)$. Hence we can choose $\eta=\varphi_{r_1,r_2}$ for some $0<r_1<r_2$ obtaining
\begin{equation}\label{sob_vol}
V(B_{r_1})^{q/p}\leq \frac{S_M^q}{\left(\int_{r_1}^{r_2}a_q(s)ds\right)^{q-1}}.
\end{equation}
In particular, letting $r_2\to\infty$ gives $a_q\in L^1(0,+\infty)$, since otherwise $V(B_{r_1})\equiv 0$ for all $r_1>0$. Even if this conclusion is immediately implied by the ($q$-)hyperbolicity of $M$, we can combine inequality (\ref{sob_vol}) with Carron's estimate (\ref{carron})  to obtain a slightly improved result.
\begin{proposition}\label{prop_la}
Let $M$ be a $m$-dimensional complete non-compact manifold supporting the euclidean Sobolev inequality \eqref{sob} for some $q<p$ and $p=mq/(m-q)$. Then there exists a positive constant
$0<C_S=\gamma^{-({m-q})/({m(q-1)})}S_M^{q/({q-1})}$ such that
\begin{equation}\label{lower_area}
r^{(m-q)/({q-1})}\int_r^{\infty}a_q(s)ds \leq C_S \qquad \text{for all } r>0.
\end{equation}
\end{proposition}
\begin{remark}\rm{
We underline that Proposition \ref{prop_la} is non-trivial, in the sense that, in the absence of the validity of (\ref{sob}), there exist manifolds satisfying \eqref{carron} and $a_q\in L^1(0,+\infty)$, for which \eqref{lower_area} does not hold. For instance, for fixed $1<q<(m+1)/{2}$, an example is given by the $m$-dimensional model manifold $(t,\theta)\in M=(0,+\infty)\times \mathbb S^{m-1}$ endowed with the Riemannian metric $\left\langle , \right\rangle_M = dt^2+h^2(t)d\theta^2$, with warping function $h\in C^{\infty}((0,+\infty))$ chosen such that
\[
\begin{cases}
h(0)=0,\qquad h'(0^+)=1,&\\
h(t)\geq t^{\beta},&\\
\left.h(t)\right|_{\left[4k+3,4k+4\right]}\equiv t^{\beta}& \text{for } k=0,1,2,\dots,\\
\left.h(t)\right|_{\left[4k+1,4k+2\right]}\equiv Ht& \text{for }k=0,1,2,\dots
\end{cases}
\]
for some constants
\[
\frac{q-1}{m-1}<\beta<\frac{m-q}{m-1}<1\qquad\textrm{and}\qquad H>\frac 1 4 \left(\frac{m10^m\gamma}{A(\partial \mathbb B_1^{\mathbb R^m})}\right)^{1/({m-1})}.
\]}
\end{remark}

\end{document}